\documentclass{amsart}

\usepackage{bbold,psfrag,graphicx,centernot,array,comment,float}

\newcommand{\io}{\text{ i.o.}}
\newcommand{\snk}{S_{n_k}}
\newcommand{\snl}{S_{n_l}}
\newcommand{\Var}{\mathrm{Var}}
\newcommand{\Cov}{\mathrm{Cov}}
\newcommand{\galton}{\genfrac{[}{]}{0pt}{}}

\newcommand{\ER}{ER}
\newcommand{\KS}{KS}
\newcommand{\D}{D}
\newcommand{\PWI}{PWI}
\newcommand{\SUB}{SUB}
\newcommand{\NOP}{NOP}
\newcommand{\IND}{IND}
\newcommand{\B}{B}
\newcommand{\IO}{IO}

\newtheorem{theorem}{Theorem}[section]
\newtheorem{lemma}[theorem]{Lemma}

\author{Csaba Bir\'o}
\email{csaba.biro@louisville.edu}

\author{Israel R.~Curbelo}
\email{israel.curbelo@louisville.edu}

\address{Department of Mathematics, University of Louisville, Louisville, KY 40292}

\title[Strong converses of the Borel--Cantelli Lemma]{Weak independence of events and the converse of the Borel--Cantelli Lemma}

\begin{document}

\begin{abstract}
The converse of the Borel-Cantelli Lemma states that if $\{A_i\}_{i=1}^\infty$ is a sequence of independent events such that $\sum P(A_i)=\infty$, then almost surely infinitely many of these events will occur. Erd\H os and R\'enyi proved that it is sufficient to weaken the condition of independence to pairwise independence. Later, several other weakenings of the condition appeared in the literature. The aim of this paper is to provide a collection of conditions, all of which imply that almost surely infinitely many of the events occur, and determine the complete implicational relationship between them.

Many of these results are known, or follow from known results, however, they are not widely known among non-specialists. Yet, the results can be extremely useful for areas outside of probability theory, as evidenced by the original motivation of this paper emerging from infinite combinatorics. Our proofs are aimed to be accessible to a general mathematical audience.
\end{abstract}

\maketitle

\section{Introduction}

Let $(\Omega,\mathcal{F},P)$ be a probability space, and let $\{A_n\}_{n=1}^\infty$ be a sequence of events in $\mathcal{F}$. We let
\[
\limsup_n A_n = \lim_{m\to\infty} \bigcup_{n=m}^{\infty} A_n=\{\omega\in\Omega:\omega \text{ is in infinitely many of the $A_i$'s}\},
\]
and
\[
\liminf_n A_n = \lim_{m\to\infty} \bigcap_{n=m}^{\infty} A_n=\{\omega\in\Omega:\omega \text{ is in all but finitely many of the $A_i$'s}\}.
\]
Note that these limits exist in arbitrary measure spaces, because the sets in question are monotone. It is common to write $\limsup A_n = \{\omega\in\Omega:\omega\in A_n\io\}$, and $P(\limsup A_n)=P(A_n\io)$ where i.o. stands for infinitely often.

The classical Borel--Cantelli Lemma states the following.

\begin{theorem}
If $\sum_{n=1}^{\infty} P(A_n)<\infty$ then $P(A_n \io)=0.$
\end{theorem}

The converse is obviously not true without adding some conditions. The most well-known of these is independence of these events. The resulting theorem is what usually is called the Second Borel--Cantelli Lemma.

\begin{theorem}
If $\sum_{n=1}^{\infty} P(A_n)=\infty$ and the events $\{A_n\}$ are independent then $P(A_n \io)=1.$
\end{theorem}

The condition of independence is very strong, so effort has been made to weaken it or replace it entirely. In 1959 Erd\H os and R\'enyi \cite{Erd-Ren-59} showed that instead of independence, one can assume pairwise independence. In fact they proved an even stronger theorem.

\begin{theorem}
Let $A_n$ be a sequence of events such that $\sum_{n=1}^{\infty} P(A_n)=\infty$ and
\begin{equation}\label{eq:ercondition}
    \liminf_{n\to\infty}\frac{\sum_{k=1}^n\sum_{l=1}^n P(A_k A_l)}{(\sum_{k=1}^n P(A_k))^2}=1.
\end{equation}
then $P(A_n \io)=1.$
\end{theorem}

In their paper they point out that nonpositive correlation implies their condition. (Pairwise independence of the events is a special case of nonpositive correlation.)

This result was rediscovered or improved several times by different authors: see e.g.\ Lamperti~\cite{Lam-63}, Kochen and Stone~\cite{Koc-Sto-64}, and Spitzer~\cite{Spi-PORW}. The latter two also found lower bounds for $P(A_n\io)$ in terms of the left hand side of (\ref{eq:ercondition}). M\'ori and Sz\'ekely~\cite{Mor-Sze-83} improved on these bounds. Some of our conditions were inspired by these papers.

The research in this area continues. Frolov~\cite{Fro-12} made improvements on lower bounds on $P(A_n\io)$, thereby generalizing earlier theorems. A monograph by Chandra \cite{Cha-TBCL} provides a good overview of the subject.

In \cite{Dur-PTAE} the following theorem is proven.

\begin{theorem}
Let $A_n$ be a pairwise independent sequence of events such that $\sum_{n=1}^{\infty} P(A_n)=\infty$. Let $I_n=\mathbb{1}_{A_n}$ be the indicator variable of $A_n$. Then, as $n\to\infty$
\begin{equation}\label{eq:durrett}
    \frac{\sum_{i=1}^n I_i}{\sum_{i=1}^n P(A_i)}\to 1 \quad\text{a.s.}
\end{equation}
\end{theorem}

Obviously (\ref{eq:durrett}) implies $P(A_n \io)=1.$ We later prove that an even weaker condition implies $P(A_n \io)=1.$

The original motivation of this paper came from studying infinite random graphs \cite{Bir-Dar-18}. In this paper, a result by Bruss \cite{Bru-80} is used. This theorem is usually called the/a counterpart of the Borel--Cantelli Lemma.

\begin{theorem}
Let $A_n$ be a sequence of events such that $A_k\subseteq A_{k+1}$, and
\[
\sum_{n=1}^\infty P(A_{n+1}|\overline{A_n})=\infty.
\]
where $\overline{A}$ denotes the complement of $A$. Then $P(A_n \io)=1$.
\end{theorem}

In this theorem, it is usually assumed that $P(A_n)\neq 1$, so the conditional probabilities are well defined. However, to make the theorem more elegant and widely applicable, we will allow $P(A_n)=1$, and we define $P(A_{n+1}|\overline{A_n})=1$ for that case. We will use this convention for the balance of the paper.

Despite having a simple proof, our results show that one can use Bruss's Theorem to prove some very general versions of the converse of the Borel--Cantelli Lemma.

\subsection{The structure of this paper}

In Section~\ref{sec:statements} we state numerous conditions on the sequence of events $\{A_n\}$. All of these imply that almost surely infinitely many of the events $A_n$ occur, which is also the last of the conditions. These conditions can be regarded as various weakenings of the condition that the events $A_n$ are independent; these are what we referenced in the title as ``weak independence''. Most of these conditions appear in some form in various papers cited above, although we changed some of them slightly (typically making them slightly weaker, so the theorems are slightly stronger).

In this section we state our results in a single diagram (Figure~\ref{fig:diag}). We completely determined the implicational relationship of the conditions. The rest of the paper is dedicated the proof of every implication and non-implication.

Though many of these proofs are known, we decided that for completeness' sake we include them all. In some cases we slightly generalized earlier results; in other cases we simplified proofs or proved known results in different, more direct ways.

But perhaps the most interesting part of the paper is the part on the negative results. In particular, the longest and most complicated result is the proof that $\text{\D}\centernot\implies\text{\ER}$. This proof introduces a tool we called ``Galton sequences'', which can be used to construct examples for sequences of events satisfying various conditions.

Although many ideas and proofs in the paper are original, our main contribution to the subject is expository in nature.

\section{Conditions, statements of results}\label{sec:statements}

In the remainder of the paper, we will deal with an infinite sequence of events $A_1,A_2,\ldots$ of some probability space $(\Omega,\mathcal{F},P)$. Since we are interested in the converse statements of the Borel--Cantelli Lemma, we will always assume $\sum_{n=1}^\infty P(A_n)=\infty$. We will use $I_n=\mathbb{1}_{A_n}$ for the indicator variable of the event $A_n$: that is $I_n=1$ if $A_n$ occurs, $I_n=0$ if it does not. We will use $S_n=\sum_{i=1}^n I_i$ for the number of events that occurred in the first $n$ events, and $\mu_n=E[S_n]$ for the expected number of occurred events. We will use $X_n=S_n/\mu_n$ if $\mu_n\neq 0$, otherwise we set $X_n=1$. Notice that $E[X_n]=1$. Finally, we use $E_m^n=\cup_{i=m}^n A_i$.

As customary, we use $\xrightarrow{d}$, $\xrightarrow{p}$, and $\xrightarrow{a.s.}$ for convergence of random variables in distribution, probability, and almost surely, respectively.

We also use the term ``eventually'' to express that that there exits a positive integer $N$ such that the condition is satisfied whenever all indices in the condition are at least $N$.

\bigskip

\begin{center}
\begin{tabular}{cc}
\textbf{\IND} & $A_n$ are independent (eventually)\\
\textbf{\PWI} & $A_n$ are pairwise independent (eventually)\\
\textbf{\NOP} & for all $i,j$, $\Cov(I_i,I_j)\leq 0$ (eventually)\\
\textbf{\ER} & $\displaystyle\liminf_{n\to\infty} E[X_n^2]=1$\\
\textbf{\KS} & $\displaystyle\limsup_{n\to\infty} E[S_n]^2/E[S_n^2]=1$\\
\textbf{\D} & $X_n\xrightarrow{a.s.} 1$\\
\textbf{\SUB} & $\exists$ subsequence $X_{n_k}$ of $X_n$ such that $X_{n_k}\xrightarrow{p}1$ (equivalently, $X_{n_k}\xrightarrow{d}1$)\\
\textbf{\B} & for all $m$, $\displaystyle\sum_{i=m}^\infty P(E_{m}^{i+1}|\overline{E_m^i})=\infty$\\
\textbf{\IO} & almost surely $A_i$ occurs infinitely often\\
\end{tabular}
\end{center}

We completely determined the quasiorder of implications of these conditions. Figure~\ref{fig:diag} contains all the information. A single arrow is implication, and a box contains equivalent conditions.

\begin{figure}[H]
\psfrag{er}{\ER}
\psfrag{ks}{\KS}
\psfrag{d}{\D}
\psfrag{pwi}{\PWI}
\psfrag{sub}{\SUB}
\psfrag{nop}{\NOP}
\psfrag{ind}{\IND}
\psfrag{b}{\B}
\psfrag{io}{\IO}
\includegraphics[scale=0.5]{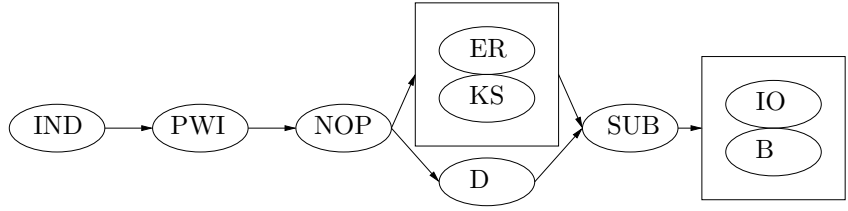}
\caption{Diagram of implications}\label{fig:diag}
\end{figure}

\section{Proofs and examples}

This section will provide all proofs and counterexamples that justifies the diagram in Figure~\ref{fig:diag}.

Note that none of the conditions are influenced by adding or removing finitely many events to the sequence $\{A_i\}$. Suppose $\mu_n=0$ for some $n$. The condition $\mu_n\to\infty$ and that $\mu_n$ is monotone increasing imply that this can only happen for the first finitely many values of $n$. In this case, we can remove those first events from the beginning of the sequence without changing the truth value of any of the conditions. So in the proofs, we may assume that $\mu_n>0$, and $X_n=S_n/\mu_n$ for all $n>0$.

\subsection{Equivalences}

\subsubsection{$\text{\ER}\iff\text{\KS}$}\ 

Note that $E[X_n^2]=E[S_n^2/\mu_n^2]=\frac{E[S_n^2]}{E[S_n]^2}$, hence
\[
\liminf_{n\to\infty}E[X_n^2]=\limsup_{n\to\infty}\frac{1}{E[X_n^2]}=\limsup_{n\to\infty}\frac{E[S_n]^2}{E[S_n^2]}.
\]

\subsubsection{$\text{\IO}\iff\text{\B}$}\ 

The main idea here is to recognize that we can ``translate'' between general events $A_i$, and increasing sets of events $E_m^i$, making it possible to use Bruss's Theorem. Note that for the events $E_m^i$, one occurs if and only if infinitely many occurs, and this ``all or nothing'' phenomenon appears in this section in an obvious way.

The following lemma is essentially an adopted version of Bruss's proof.

\begin{lemma}\label{lem:bruss}
Let $m$ be a positive integer. Then $P(E_m^i\io)=1$ if and only if $\sum P(E_m^{i+1}|\overline{E_m^i})=\infty$.
\end{lemma}

\begin{proof}
If there exists an $i$ such that $P(E_m^i)=1$, then clearly both conditions of the lemma hold (recall our convention on conditioning on events of probability zero). So from now, we assume that $P(E_m^i)<1$ for all $i$.

Let $q_m=\lim_{i\to\infty}P(\overline{E_m^{i}})$, the probability that none of $E_m^i$ occurs. Note that $P(E_m^i\io)=1$ means exactly $q_m=0$.
\begin{multline*}
P(\overline{E_m^i})=
P(\overline{E_m^i}|\overline{E_m^{i-1}})
P(\overline{E_m^{i-1}}|\overline{E_m^{i-2}})
\cdots
P(\overline{E_m^{m+1}}|\overline{E_m^m})
P(\overline{E_m^m})=\\
P(\overline{E_m^m})\prod_{j=m}^{i-1}P(\overline{E_m^{j+1}}|\overline{E_m^j})
\end{multline*}
Now let $i\to\infty$, to get
\[
q_m=
P(\overline{E_m^m})\prod_{j=m}^\infty P(\overline{E_m^{j+1}}|\overline{E_m^j})=
P(\overline{E_m^m})\prod_{j=m}^\infty (1-P(E_m^{j+1}|\overline{E_m^j})).
\]
By our assumption $P(\overline{E_m^m})>0$, so $q_m=0$ if and only if $\prod (1-P(E_m^{j+1}|\overline{E_m^j}))=0$, which is equivalent to $\sum P(E_m^{j+1}|\overline{E_m^j})=\infty$.
\end{proof}

Now notice that $\{E_m^i\io\}=\bigcup_{i=m}^\infty A_i$ and by definition $\{A_i\io\}=\bigcap_{m=1}^\infty \bigcup_{i=m}^\infty A_i.$ Thus the result follows immediately from Lemma~\ref{lem:bruss} and the fact that 
\[
P\left(\bigcup_{i=m}^\infty A_i\right)=1 \text{ for all $m$} \iff P\left(\bigcap_{m=1}^\infty \bigcup_{i=m}^\infty A_i\right)=1.
\]

\subsection{Positive results}

\subsubsection{$\text{\IND}\implies\text{\PWI}$} Trivial.

\subsubsection{$\text{\PWI}\implies\text{\NOP}$} Trivial.

\subsubsection{$\text{\NOP}\implies\text{\ER}$}\label{sec:noptoer}\

First we will show that if $\Cov(I_i,I_j)\leq 0$ for all $i,j\geq 1$, then in fact 
\[
\lim_{n\to\infty}E[X_n^2]=1.
\]
The case, when the covariance condition is false for some pairs of events, will be dealt later. Notice, however, that the conclusion is stronger than required.

With the now stronger condition, we can now state simple upper bounds for the variance of $S_n$ and $X_n$. Since we will need this in a later proof as well, we state this as a lemma.
\begin{lemma}\label{lem:xnupper} Suppose $\Cov(I_i,I_j)\leq 0$ for all $i,j\geq 1$. Then
\[
\Var(S_n)\leq \mu_n;\qquad\Var(X_n)\leq\frac{1}{\mu_n}.
\]
\end{lemma}
\begin{proof}
\begin{multline*}
\Var(S_n)=
\Var(\sum_{i=1}^n I_i)=
\sum_{i=1}^n\Var(I_i)+2\underset{1\leq i<j\leq n}{\sum\sum} \Cov(I_i,I_j)\\
\leq\sum_{i=1}^n\Var(I_i)
=\sum_{i=1}^n P(A_i)(1-P(A_i))
\leq \sum_{i=1}^n P(A_i)
=\mu_n
\end{multline*}
\[
\Var(X_n)=\Var\left(\frac{S_n}{\mu_n}\right)=\frac{\Var(S_n)}{\mu_n^2}
\leq \frac{1}{\mu_n}
\]
\end{proof}

On the other hand, $0\leq\Var(X_n)=E[X_n^2]-E[X_n]^2=E[X_n^2]-1$, so $E[X_n^2]\geq 1$. Putting this together with the upper bound,
\[
1\leq E[X_n^2]=1+\Var(X_n)\leq 1+\frac{1}{\mu_n}.
\]
Then we let $n\to\infty$, which makes $\mu_n\to\infty$, and therefore $E[X_n^2]\to 1$.

Now suppose that the covariance condition can fail for some pairs of events. Let $N$ be an positive integer such that $\Cov(I_i,I_j)\leq 0$ for all $i,j\geq N$. We will define events $A_1',A_2',\ldots$ as follows: $A_i'=\emptyset$ for $i<N$, and $A_i'=A_i$ for $i\geq N$. Define $I_i'$, $S_n'$, $\mu_n'$, $X_n'$ accordingly. This time $\Cov(I_i',I_j')\leq 0$ for all $i,j\geq 1$. Notice that
\begin{equation}\label{eq:primebounds}
S_n-N\leq S_n'\leq S_n,\quad\text{and}\quad \mu_n-N\leq\mu_n'\leq\mu_n,
\end{equation}
and hence $\mu_n'\to\infty$. Accordingly, we apply the argument above. We conclude
\[
\lim_{n\to\infty}E[(X_n')^2]=1.
\]
But $X_n$ and $X_n'$ are not so different:
\[
\left(\frac{S_n'}{\mu_n'+N}\right)^2
\leq X_n^2=\left(\frac{S_n}{\mu_n}\right)^2
\leq\left(\frac{S_n'+N}{\mu_n'}\right)^2
\]
To take expectations, we note
\[
E\left[\left(\frac{S_n'}{\mu_n'+N}\right)^2\right]=
E\left[\left(\frac{X_n'}{1+\frac{N}{\mu_n'}}\right)^2\right]=
\frac{1}{\left(1+\frac{N}{\mu_n'}\right)^2}
E\left[(X_n')^2\right],
\]
and
\[
E\left[\left(\frac{S_n'+N}{\mu_n'}\right)^2\right]=
E\left[\left(X_n'+\frac{N}{\mu_n'}\right)^2\right]=
E[(X_n')^2]+2\frac{N}{\mu_n'}E[X_n']+\frac{N^2}{(\mu_n')^2}.
\]
Using that $\mu_n'\to\infty$, and $E[(X_n')^2]\to 1$, we deduce that both of these expectations converge to $1$, and hence $E[X_n^2]\to 1$.

\emph{Remark: } We indeed proved a stronger result by showing that the limit of the sequence $E[X_n^2]$ is $1$, as opposed to just its $\liminf$.

\subsubsection{$\text{\NOP}\implies\text{\D}$}\

The following argument is along the line of that in \cite{Dur-PTAE}. We generalized it to prove our more general statement (with the ``eventually''); any other change is just in the exposition.

Just like in Section~\ref{sec:noptoer}, we will assume nonpositive correlation for all events first, then we modify the events to satisfy the stronger condition, and we show how the conclusion is achieved for the original events.

Following this plan, we assume $\Cov(I_i,I_j)\leq 0$ for all $i,j\geq 1$. The next step is to apply Chebyshev's inequality for the variable $S_n$. We use Lemma~\ref{lem:xnupper}. Let $\epsilon>0$.

\begin{equation}\label{eq:convprob}
P(|S_n-\mu_n|\geq \epsilon\mu_n)\leq\frac{\Var(S_n)}{\epsilon^2\mu_n^2}\leq\frac{\mu_n}{\epsilon^2\mu_n^2}=\frac{1}{\epsilon^2\mu_n}
\end{equation}
If we let $n\to\infty$, the right hand side converges to $0$, so we could use this inequality to prove $X_n\xrightarrow{p}1$. However, that is irrelevant for us, as we need almost sure convergence. The trick is to get a better bound on a subsequence first, use the (forward) Borel-Cantelli Lemma to get almost sure convergence of the subsequence, and make a deduction back to the sequence.

Let $n_k=\min\{n:\mu_n\geq k^2\}$. Let $T_k=S_{n_k}$ the mentioned subsequence, let $\nu_k=E[T_k]$, and note that $\nu_k=\mu_{n_k}$. Notice further that due to $n_k$ being a minimum,
\begin{equation}\label{eq:nukbounds}
k^2\leq \nu_k\leq k^2+1
\end{equation}

Use the inequality (\ref{eq:convprob}) for the subsequence $T_k$.
\[
P(|T_k-\nu_k|\geq \epsilon\nu_k)\leq\frac{1}{\epsilon^2\nu_k}\leq\frac{1}{\epsilon^2 k^2}
\]
Summing the inequalities for all $k$, we get
\[
\sum P(|X_{n_k}-1|\geq \epsilon)<\infty.
\]
The Borel-Cantelli Lemma now implies that a.s.\ only finitely many of the events $|X_{n_k}-1|\geq\epsilon$ occur, which exactly means that $X_{n_k}\xrightarrow{a.s.}1$. We just need to show $X_n\xrightarrow{a.s.}1$.

Let $n,k$ be such that $n_k\leq n\leq n_{k+1}$ so that $\mu_{n_k}\leq \mu_n\leq \mu_{n_{k+1}}$. The sequence $S_n$ is monotone increasing, so
\[
\frac{S_{n_k}}{\mu_{n_k}}\cdot\frac{\nu_k}{\nu_{k+1}}
\leq
\frac{S_{n_k}}{\mu_{n_k}}\cdot\frac{\mu_{n_k}}{\mu_n}
\leq
\frac{S_n}{\mu_n}
\leq
\frac{S_{n_{k+1}}}{\mu_{n_{k+1}}}\cdot\frac{\mu_{n_{k+1}}}{\mu_n}
\leq
\frac{S_{n_{k+1}}}{\mu_{n_{k+1}}}\cdot\frac{\nu_{k+1}}{\nu_k}.
\]
The leftmost side can be written as
\[
\frac{T_k}{\nu_k}\cdot\frac{\nu_k}{\nu_{k+1}}=X_{n_k}\cdot\frac{\nu_k}{\nu_{k+1}},
\]
and similarly, the rightmost side is
\[
X_{n_{k+1}}\cdot\frac{\nu_{k+1}}{\nu_k}.
\]
If we can prove $\nu_k/\nu_{k+1}\to 1$ (or equivalently, $\nu_{k+1}/\nu_k\to 1$), then taking limits in the inequality above as $k\to\infty$, we will conclude $X_n\xrightarrow{a.s.} 1$, as intended.

It remains to prove $\nu_k/\nu_{k+1}\to 1$. By the bounds (\ref{eq:nukbounds}),
\[
\frac{k^2}{(k+1)^2+1}\leq\frac{\nu_k}{\nu_{k+1}}\leq\frac{k^2+1}{(k+1)^2}.
\]
Clearly, both ends converge to $1$, which finishes our proof, at least in the case when the covariance condition is true for pairs of events.

For the balance, assume there exists a positive integer $N$ such that $\Cov(I_i,I_j)\leq 0$ whenever $i,j\geq N$. Define a sequence of events $A_1',A_2',\ldots$, just like in Section~\ref{sec:noptoer}: $A_i'=\emptyset$ for $i<N$, and $A_i'=A_i$ for $i\geq N$. Again, define $I_i'$, $S_n'$, $\mu_n'$, $X_n'$ accordingly. The bounds (\ref{eq:primebounds}) are still true, and by the argument above, $X_n'\xrightarrow{a.s.}1$. Therefore
\[
\frac{X_n'}{1+\frac{N}{\mu_n'}}=
\frac{S_n'}{\mu_n'+N}\leq
\frac{S_n}{\mu_n}\leq
\frac{S_n'+N}{\mu_n'}=
X_n'+\frac{N}{\mu_n'}.
\]

Both ends almost surely converges to $1$, so we conclude $X_n\xrightarrow{a.s.}1$, finishing the proof.

\subsubsection{$\text{\ER}\implies\text{\SUB}$}\ 

This is immediate from a simple lemma that is probably known. In fact stronger versions can be proven, but we only present a proof of what we need.

\begin{lemma}\label{lem:vartoconv}
If $\{X_n\}$ is a sequence of random variables such that $E[X_n]=\mu$ for all $n$, and $\Var(X_n)\to 0$ then $X_n\xrightarrow{p}\mu$.
\end{lemma}

\begin{proof}
Let $\epsilon,\delta>0$. Let $k$ be such that $1/k^2<\delta$, and let $N$ be such that for all $n\geq N$, $k\sigma_n<\epsilon$, where $\sigma_n=\sqrt{\Var(X_n)}$, the standard deviation of $X_n$. Then, by Chebyshev's inequality, for all $n\geq N$,
\[
P(|X_n-\mu|\geq\epsilon)\leq P(|X_n-\mu|\geq k\sigma_n)\leq\frac{1}{k^2}<\delta.
\]
\end{proof}

To prove the implication of this section, we suppose $\liminf E[X_n^2]=1$. Then there is a subsequence $X_{n_k}$ such that $E[X_{n_k}^2]\to 1$, or, equivalently, $\Var(X_{n_k})\to 0$. Then, applying Lemma~\ref{lem:vartoconv} for $X_{n_k}$, we conclude $X_{n_k}\xrightarrow{p}1$.

\subsubsection{$\text{\D}\implies\text{\SUB}$}\ 

Since almost sure convergence implies convergence in probability, and convergence of a sequence implies convergence of a subsequence, this is trivial.

\subsubsection{$\text{\SUB}\implies\text{\IO}$}\ 

Assume that $X_{n_k}$ is a subsequence with $X_{n_k}\xrightarrow{p}1$. Then there exists a sub-sequence $\snl$ of $\snk$ such that 
\begin{equation}\label{eq:4}
    P\left(\left|\frac{\snl}{E[\snl]}-1\right|>\frac{1}{2}\right)<2^{-l}.
\end{equation}
Notice that 
\begin{equation}\label{eq:5}
    \left\{\snl<\frac{1}{2}E[\snl]\right\}\subseteq \left\{|\snl - E[\snl]|>\frac{1}{2}E[\snl]\right\}=\left\{\left|\frac{\snl}{E[\snl]}-1\right|>\frac{1}{2}\right\}.
\end{equation}
Then from (\ref{eq:4}) and (\ref{eq:5}) we get

\[
    P(\snl<E[\snl]/2)<P\left(\left|\frac{\snl}{E[\snl]}-1\right|>1/2\right)<2^{-l}.
\]
Taking sums on both sides gives us that
\[
    \sum_{l=1}^{\infty}P(\snl<E[\snl]/2)<\infty .
\]
By the Borel-Cantelli lemma 
\[
    P(\snl<E[\snl]/2 \io) = 0.    
\]
Hence $\snl\geq E[\snl]/2$ for all but finitely many $l$ almost surely. Since we assumed that $E[S_n]\to\infty$, $\snl\to\infty$ almost surely. Thus $P(A_n \io)=1.$

\subsection{Negative results}

\subsubsection{$\text{\PWI}\centernot\implies\text{\IND}$} Trivial.

\subsubsection{$\text{\NOP}\centernot\implies\text{\PWI}$} Trivial.

\subsubsection{$\text{\ER}\centernot\implies\text{\D}$}\ 

Consider the probability space $((0,1], B, \lambda)$ where $B$ is the Borel $\sigma$-field and $\lambda$ is Lebesgue measure. Now consider a sequence $A_1, A_2,...$ consisting of the two events $(0,1]$ and $(0,\tfrac{1}{2}]$. The sequence will consist of alternating runs of $(0,1]$'s and $(0,\tfrac{1}{2}]$'s as such
\[
(0,1],..., (0,1], (0,\tfrac{1}{2}],..., (0,\tfrac{1}{2}], (0,1],..., (0,1], (0,\tfrac{1}{2}],..., (0,\tfrac{1}{2}],..
\]
Let $a_i$ be the length of the $i$th run of $(0,1]$'s. Then the $i$th run of $(0,\tfrac{1}{2}]$'s will also be of length $a_i$.

Then $\mu_n=\frac{3}{2}(a_1+\cdots+a_{i-1})+a_i$. So $S_n$ is either $2(a_1+\cdots+a_{i-1})+a_i$, or $a_1+\cdots+a_{i-1}+a_i$. Therefore
\[
E[X_n^2]=\frac{
\left(\frac{2(a_1+\cdots+a_{i-1})+a_i}{\frac{3}{2}(a_1+\cdots+a_{i-1})+a_i}\right)^2+
\left(\frac{(a_1+\cdots+a_{i-1})+a_i}{\frac{3}{2}(a_1+\cdots+a_{i-1})+a_i}\right)^2}
{2}
\]
As $a_i\to\infty$, we have $E[X_n^2]\to 1$. So it is possible to choose $a_i$ to be the length needed so that if $A_n$ is the last $(0,1]$ term of the $i$th run, then $E[X_n^2]-1<2^{-n}$. Thus \ER\ is satisfied.

To see that \D\ is not satisfied, notice that if $A_{i_1}, A_{i_2},...$ are the last terms of the $(0,\tfrac{1}{2}]$ runs, then for $j=1,2,... $, $X_{i_j}=\tfrac{4}{3}$ when $\omega\in (0,\tfrac{1}{2}]$ and $X_{i_j}=\tfrac{2}{3}$ when $\omega\in (\tfrac{1}{2},1]$.

\subsubsection{$\text{\D}\centernot\implies\text{\ER}$}\

We start by fixing a function whose domain is a set of pairs of nonnegative integers $(n,k)$ with $0\leq k\leq n$, and codomain is the interval $[0,1]$. We will use the notation $\galton{n}{k}$ for the value of this function at $(n,k)$, and we will refer to the values as Galton coefficients. Our state space $\Omega$ will be the set of infinite $0$--$1$ sequences, and the event $A_i$ is that the $i$th digit is $1$. We will start indexing the digits and the events with $i=0$. The probability measure will be based on the Galton coefficients. We will refer to these kind of sequences of events in these probability spaces as Galton sequences\footnote{We named these sequences after Sir Francis Galton, inventor of the Galton board, which is a physical device used to illustrate the binomial distribution. In case $\galton nk=p$ for all $n,k$, the distribution of $S_n$ is binomial.}. After making some general observations about Galton sequences, we will show how to set the Galton coefficients to generate a Galton sequence for which \D\ holds, but \ER\ does not. The reason we start indexing the events with $0$ as opposed to $1$ is convenience of notation; we adjust the other notations accordingly. This, of course, does not change behavior in the limit.

After fixing Galton coefficients, here is how to define the Galton sequence. We will describe how to recursively sample the events.
\begin{itemize}
\item $A_0$ occurs with probability $\galton00$.
\item If $A_0$ did not occur, then $A_1$ occurs with probability $\galton01$, otherwise it occurs with probability $\galton11$.
\item In general, suppose that $A_0,A_1,\ldots,A_{i-1}$ is determined. Then $A_i$ occurs with probability $\galton i k$, where $k$ is the number of events occurred among $\{A_0,\ldots,A_{i-1}\}$.
\end{itemize}

The definitions of $S_n$ must be modified to count events among $\{A_0,\ldots,A_n\}$, but all other notations are essentially unchanged. It is still clear that $E[X_n]=1$ for all $n\geq 0$.

In the next step, we will show how to define Galton coefficients such that $\mu_n\to\infty$, $X_n\xrightarrow{a.s}1$, but $E[X_n^2]\to\infty$, providing the required example. In the remaining of this section, we use $\log$ to denote logarithm of base $2$.

\[
\galton{n}{k}=\begin{cases}
1\text{ if $n+1$ is a power of $2$, and $k=\lfloor\log n\rfloor$;}\\
\frac13\text{ if $n$ is a power of $2$, and $k=n$;}\\
1\text{ if $n$ is not a power of $2$, and $k=n$;}\\
0\text{ otherwise.}
\end{cases}
\]
(See Figure~\ref{fig:galton}.)

\begin{center}
\begin{figure}
\includegraphics[scale=0.5]{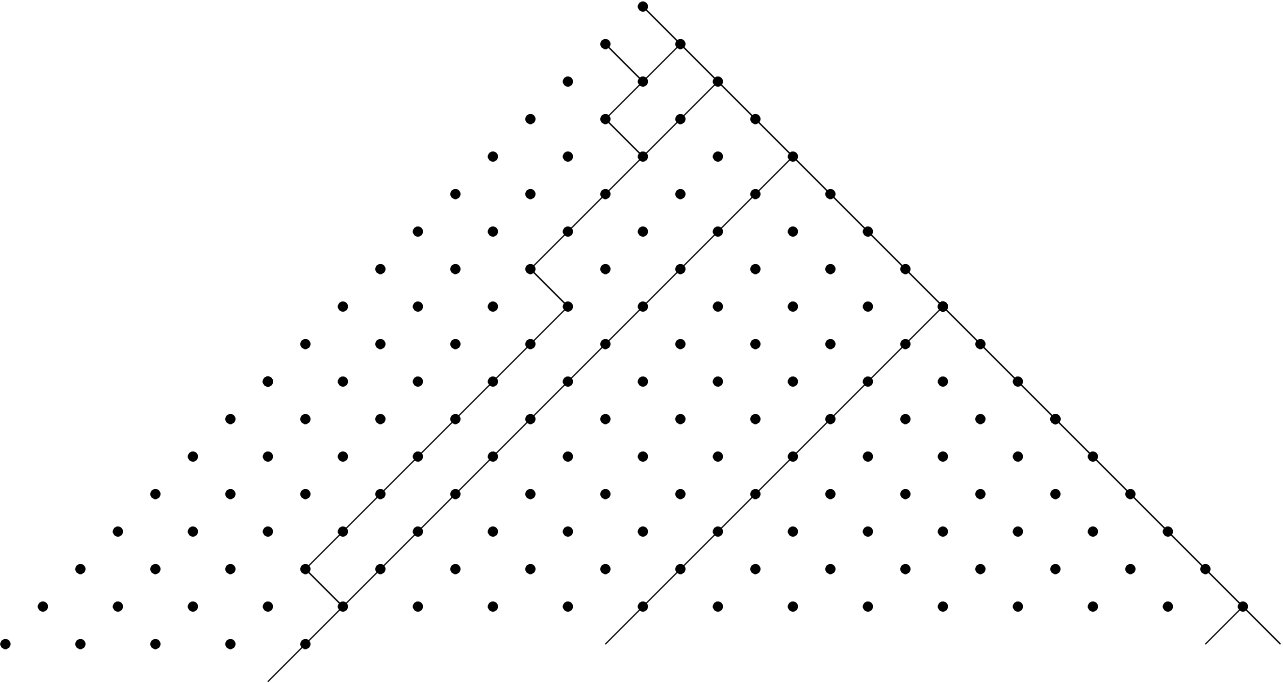}
\caption{Illustration of Galton coefficients. The lines denote the possible paths of the balls in a physical Galton board. The leftmost broken line describes a logarithmic ``curve''. Any time lines branch, the probabilities of a ball going right is $1/3$.}\label{fig:galton}
\end{figure}
\end{center}

The first goal is to determine the distribution of $S_n$. This is used to to compute $\mu_n$, and $E[X_n^2]$. The probability mass function (p.m.f.) $p_n(i)$ of $S_n$ can be used (with the Galton coefficients) to determine the p.m.f.\ $p_{n+1}(i)$ of $S_{n+1}$. Indeed, 
\[
p_{n+1}(i)=p_n(i-1)\galton{n}{i-1}+p_n(i)\left(1-\galton{n}{i}\right).
\]
Now let $n\geq 2$, $k=\lfloor\log n\rfloor$, $m=\lceil\log n\rceil$, and let $l$ be the least power of $2$ that is greater than $k$.
One can use induction to show
\[
p_n(i)=\begin{cases}
1-1/3^{\log l}\text{ if }i=k\\
2/3^{1+\log l}\text{ if }i=l\\
2/3^{2+\log l}\text{ if }i=2l\\
\vdots\\
2/3^{m}\text{ if }i=2^{m-\log l-1}l=2^{\lceil\log n\rceil-1}\\
1/3^{m}\text{ if }i=n\\
0\text{ for all other values of $i$}.
\end{cases}
\]

Right away it follows that $\mu_n\geq k$, and so $\mu_n\to\infty$. On the other hand,
\[
\mu_n=\left(1-\frac{1}{3^{\log l}}\right)k+
\sum_{i=1}^{m-\log l}\frac{2}{3^{i+\log l}}2^{i-1}l+
\frac{1}{3^{m}}n.
\]
The middle summation term is the sum of a geometric sequence with first term $(2l)/3^{1+\log l}$, and quotient $2/3$. This means that regardless of the number of terms, it is upper bounded by $(2l)/3^{\log l}$. Using this, and other simple inequalities,
\begin{equation}\label{eq:expbound}
\mu_n\leq k+\frac{2l}{3^{\log l}}+\frac{n}{3^m}\leq k+1,
\end{equation}
if $n$ is large enough, because both the second and the third term converges to zero.

The first goal is to show $X_n\xrightarrow{a.s.}1$. In this part of the proof, we write $k=k(n)=\lfloor\log n\rfloor$ to emphasize the dependence of $k$ on $n$. Let $\omega\in\Omega$ be a $0$--$1$ sequence.

From the definition it is immediate that
\begin{itemize}
\item if $S_n(\omega)=k(n)$, then $S_i(\omega)=k(i)$ for all $i\geq n$;
\item if $S_n(\omega)<n$ then there exists $N$ (which depends on $\omega$) such that $S_N(\omega)=k(N)$, and therefore $S_i(\omega)=k(i)$ for all $i\geq N$.
\end{itemize}

Let $K=\{\omega\in\Omega: S_n(\omega)=k(n)\text{ for some $n$}\}$. Then $\overline{K}=\{\omega\in\Omega: S_n(\omega)=n\text{ for all $n$}\}$. Since $P(\overline{K})=(1/3)(1/3)\cdots=0$, we get $P(K)=1$. In other words, $P(S_n=k(n) \text{ eventually})=1$. Therefore
\[
\frac{k(n)}{k(n)+1}\leq\frac{S_n}{\mu_n}\leq\frac{k(n)}{k(n)}=1
\]
for large enough $n$, a.s. We conclude $X_n\xrightarrow{a.s.}1$.

It remains to be shown that the second moment of $X_n$ converges to infinity. For large $n$, using the p.m.f.\ of $S_n$ and (\ref{eq:expbound}),
\[
E[X_n^2]=\frac{E[S_n^2]}{\mu_n^2}\geq
\frac{\left(1-\frac{1}{3^{\log l}}\right)k^2+\frac{1}{3^{m}}n^2}{(k+1)^2}=
\frac{k^2}{(k+1)^2}-\frac{k^2}{3^{\log l}(k+1)^2}+\frac{n^2}{3^m (k+1)^2}.
\]
The first term converges to $1$, and the second one converges to $0$. So it comes down to the third term.
\[
\frac{n^2}{3^m (k+1)^2}\geq
\frac{n^2}{3^{1+\log n} (\log n+1)^2}=
\frac13\frac{n^{2-\log 3}}{(\log n+1)^2}\to\infty.
\]

\emph{Remark 1.} The technical details might hide the fact just how delicate was the process of finding this counterexample. Even after coming up with the general idea of what kind of p.m.f.\ $S_n$ needs to have, and how to construct it from events with Galton sequences, it required a very careful balancing act to make sure $p_n(n)\to 0$ just the right way. If it converges too fast, we get $\liminf E[X_n^2]=1$, and if it converges too slowly, then $\mu_n$ would get greater than $\log n$, enough to make $X_n$ converge to a number less than $1$. There is a fairly narrow band in which the rate of convergence is just right. Furthermore, a similar balancing act is necessary for $\mu_n$, which, if converges to infinity too fast, would make $\liminf E[X_n^2]=1$.

\emph{Remark 2.} Somewhat interestingly, we found that in this example $E[X_n^2]$ is \emph{not} monotone. We found a very nice proof that $E[X_n^2]\geq E[X_{n-1}^2]$, \emph{except when $n$ is a power of $2$}. This lemma would make the argument proving $E[X_n^2]\to\infty$ simpler, because we could restrict our attention to values of $n$ that are powers of $2$; however one would have to include the proof of the lemma, thereby losing the simplicity.

\subsubsection{$\text{\IO}\centernot\implies\text{\SUB}$}\ 

Consider the probability space $((0,1],B,\lambda)$ where $B$ is the Borel $\sigma$-field and $\lambda$ is Lebesgue measure. Now consider an alternating sequence of events $A_1, A_2, A_3,\ldots$ as such
\[
(0,1], (0,\tfrac{1}{2}], (0,1], (0,\tfrac{1}{2}], \ldots 
\]
Clearly $P(A_n \io)=1$. 

Since $X_n(\omega)$ converges to $4/3$ and $2/3$ for $\omega\in(0,\tfrac{1}{2}]$ and for $\omega\in(\tfrac{1}{2},1]$ respectively, we may choose $\epsilon =1/4$ so that 
\[
\lim_{n\to\infty} P(|X_n-1|>\epsilon)=1.
\]
Thus no subsequence $X_{n_k}$ of $X_n$ will converge in probability to $1$.

\bibliographystyle{plain}
\bibliography{bib}

\end{document}